\documentclass[12pt]{amsart}

\usepackage{amsmath}
\usepackage{amssymb}
\usepackage{amsthm}
\usepackage{amscd}
\usepackage{xypic}
\usepackage{delarray}

\headsep=1cm \numberwithin{equation}{section}
\newtheorem{theorem}{Theorem}[section]
\newtheorem{lemma}[theorem]{Lemma}

\theoremstyle{definition}

\newtheorem{remark}[theorem]{Remark}

\newtheorem{convention and reminder}[theorem]{Convention and Reminder}
\newtheorem{convention and remark}[theorem]{Convention and Remark}
\newtheorem{definition and remark}[theorem]{Definition and Remark}

\newtheorem{reminders and definition}[theorem]{Reminders and Definition}

\newtheorem{notation and remarks}[theorem]{Notation and Remarks}
\newtheorem{notation and remark}[theorem]{Notation and Remark}

\newcommand\depth{\operatorname{depth}}
\newcommand\codim{\operatorname{codim}}

\newcommand\Ker{\operatorname{\Ker}}

\newcommand\Sec{\operatorname{Sec}}

\begin{document}

\title[Secant loci of rational normal scrolls]
      {On varieties of almost minimal degree I : Secant loci of rational normal scrolls}

\author{M. BRODMANN and E. PARK}

\date{Z\"urich, January 2009}



\begin{abstract}
To complete the classification theory and the structure theory of varieties
of almost minimal degree, that is of non-degenerate irreducible projective
varieties whose degree exceeds the codimension by precisely 2, a natural
approach is to investigate simple projections of varieties of minimal degree.
Let $\tilde X \subset {\mathbb P}^{r+1} _K$ be a variety of minimal degree and
of codimension at least 2, and consider $X_p = \pi_p (\tilde X) \subset
{\mathbb P}^r _K$ where $p \in {\mathbb P}^{r+1} _K \backslash \tilde X$.
By \cite{B-Sche}, it turns out that the cohomological and
local properties of $X_p$ are governed by the secant locus $\Sigma_p
(\tilde X)$ of $\tilde X$ with respect to $p$.

Along these lines, the present paper is devoted to give a geometric
description of the secant stratification of $\tilde X$, that is of the
decomposition of ${\mathbb P}^{r+1} _K$ via the types of secant loci.
We show that there are exactly six possibilities for the secant locus
$\Sigma_p (\tilde X)$, and we precisely describe each stratum of the
secant stratification of $\tilde X$, each of which turns out to be a quasi-projective
variety.

As an application, we obtain the classification of all non-normal Del Pezzo varieties by providing a complete list of pairs $(\tilde X , p)$
where $\tilde X \subset {\mathbb P}^{r+1} _K$ is a variety of minimal degree, $p$ is a closed point in $\mathbb P^{r+1} _K \setminus \tilde X$ and $X_p \subset {\mathbb P}^r _K$ is a Del Pezzo variety.
\end{abstract}

\maketitle
\thispagestyle{empty}

\section{Introduction}
\label{1. Introduction}

Throughout this paper, we work over an algebraically closed field $K$ of arbitrary characteristic. We denote by ${\mathbb P}^r _K$ the projective $r$-space over $K$.

Let $X \subset {\mathbb P}^r _K$ be a non-degenerate irreducible projective variety. It is well-known that $\mbox{deg}(X) \geq \mbox{codim}(X)+1$. In case equality holds, $X$ is called a variety of minimal degree. Varieties of minimal degree were completely classified more than hundred years ago by P. Del Pezzo and E. Bertini, and now they are very well understood from several points of view. A variety $X\subset {\mathbb P}^r _K$ is of minimal degree if and only if it is either ${\mathbb P}^r _K$ or a quadric hypersurface or (a cone over) the Veronese surface in ${\mathbb P}^5 _K$ or a rational normal scroll.

In the next case, that is if $\mbox{deg}(X) = \mbox{codim}(X)+2$, one calls $X$ a variety of almost minimal degree. These varieties can be divided into two classes:
\begin{enumerate}
\item[$1.$] $X \subset {\mathbb P}^r _K$ is linearly normal and $X$ is normal;
\item[$2.$] $X \subset {\mathbb P}^r _K$ is non-linearly normal or $X$ is non-normal.
\end{enumerate}
By $(6.4.6)$ and $(9.2)$ in \cite{Fu}, $X$ is of the first type if and only if it is a normal Del Pezzo variety. Also $X$ is of the second type if and only if $X = \pi_p (\widetilde{X})$ where $\widetilde{X} \subset {\mathbb P}^{r+1} _K$ is a variety of minimal degree and of codimension at least two and $\pi_p : \widetilde{X} \rightarrow {\mathbb P}^r _K$ is the linear projection of $\widetilde{X}$ from a closed point $p$ in ${\mathbb P}^{r+1} _K \setminus \widetilde{X}$. For details, we refer the reader to Notation and Remarks \ref{2.3 Notation and Remarks}. Smooth Del Pezzo varieties are completely classified by T. Fujita (cf. $(8.11)$ and $(8.12)$ in \cite{Fu}). Also Fujita proved many important properties of singular varieties of almost minimal degree in \cite{Fu1} and \cite{Fu2}.

Now, a natural approach to understand varieties of almost minimal degree which are not normal Del Pezzo is to investigate simple projections of varieties of minimal degree. In this situation, i.e. in the case where $X = \pi_p (\widetilde{X})$, one can naturally expect that all the properties of $X$ may be precisely described in terms of the relative location of $p$ with respect to $\widetilde{X}$. For the cohomological and local properties of $X$, this expectation turns out to be true in \cite{B-Sche}. Indeed those properties are governed by the secant locus $\Sigma_p (\widetilde{X})$ of $\widetilde{X}$ with respect to $p$, which is the scheme-theoretic intersection of $\widetilde{X}$ and the union of all secant lines to $X$ passing through $p$. In other words, $\Sigma_p (\widetilde{X})$ is the correct object depending on the relative location of $p$ with respect to $\widetilde{X}$ and needed to determine the cohomological and local properties of $X$.

Along these lines, the main purpose of the present paper is to classify the
pair $(\widetilde{X},p)$ via the structure of the secant locus $
\Sigma_p (\widetilde{X})$. To this end, we first focus on solving the
following problem:\\

\begin{enumerate}
\item[(1.1)] Let $\widetilde{X} \subset {\mathbb P}^{r+1} _K$ be a
variety of minimal degree and of codimension at least two. Then classify
all the possible secant loci $\Sigma_p (\widetilde{X})$ of $\widetilde{X}$
with respect to $p$, when $p$ varies in ${\mathbb P}^{r+1} _K \setminus \widetilde{X}$.\\
\end{enumerate}

The classification theory of varieties of minimal degree says that $\widetilde{X}$ is either (a cone over) the Veronese surface in ${\mathbb P}^5 _K$ or else a rational normal scroll of degree $\geq 3$. When $\widetilde{X}$ is the Veronese surface in ${\mathbb P}^5 _K$, it is well-known that $\Sigma_p (\widetilde{X})$ is either empty or a smooth plane conic. The case where $\widetilde{X}$ is a cone over the Veronese surface can be easily dealt with from this fact. For details, see Remark~\ref{6.3 Remark}.

When $\widetilde{X}$ is a rational normal scroll of degree $\geq 3$, Theorem~\ref{thm:classificationsecantloci} and Theorem~\ref{3.3 theorem} show that there are six different possibilities for the secant locus $\Sigma_p (\widetilde{X})$. More precisely, let $\widetilde{X}$ be a cone over a smooth rational normal scroll $\widetilde{X} _0$, and let $\mbox{Vert}(\widetilde{X})$ be the set of all vertex points of $\widetilde{X}$. If $\widetilde{X}$ is smooth, i.e. $\mbox{Vert}(\widetilde{X}) = \emptyset$, then Theorem \ref{thm:classificationsecantloci} says that $\Sigma_p (\widetilde{X})$ is either empty ($=\emptyset$) or a double point ($=2q$) or the union of two simple points ($= (q_1 , q_2)$) or a smooth plane conic ($=C$) or the union of two distinct coplanar lines ($=L_1 \cup L_2$) or a smooth quadric surface ($=Q$). Moreover, for appropriate choice of the pair $(\widetilde{X},p)$ each of these six cases may be realized. When $\mbox{Vert}(\widetilde{X})$ is non-empty, Theorem~\ref{3.3 theorem} shows that $\Sigma_p (\widetilde{X})$ is the join of $\mbox{Vert}(\widetilde{X})$ with $\Sigma_{p_0} (\widetilde{X}_0)$ for some $p_0 \in \langle \widetilde{X}_0 \rangle \setminus \widetilde{X}_0$, where $\tilde X$ is a cone over the smooth scroll $\tilde X_0$.

According to this, the secant locus $\Sigma_p (\widetilde{X})$ can be of
six different types. This gives a decomposition of ${\mathbb P}^{r+1} _K \setminus \widetilde{X}$ into six disjoint strata, which will be called \textit{the secant stratification of $\widetilde{X}$}.

Our next goal is to describe this stratification in geometric terms. Suppose that $\widetilde{X}$ is smooth and let $\Phi := \{ \emptyset, 2q ,(q_1 , q_2), C , L_1 \cup L_2  ,Q \}$ be the set of symbols defined in the previous paragraph. Then Theorem~\ref{thm:classificationsecantloci} implies that \\
\begin{enumerate}
\item[(1.2)] \quad \quad \quad \quad \quad  ${\mathbb P}^{r+1} _K \setminus \widetilde{X} = \bigcup_{* \in \Phi} ~ SL^{*} (\widetilde{X})$\\
\end{enumerate}
is the secant stratification of $\widetilde{X}$ where
   \[ SL^{*} (\widetilde{X}) := \{ p \in {\mathbb P}^{r+1} _K \setminus
      \widetilde{X} ~|~ \Sigma_p (\widetilde{X}) = * \}.
   \]
As said already above, in the general case, Theorem~\ref{3.3 theorem} enables us to obtain the secant stratification of $\widetilde{X}$ from that of $\widetilde{X}_0$.

In Theorem~\ref{4.2 Theorem} and Theorem~\ref{5.3 Theorem}, we give an explicit geometric description of all the strata $SL^{*} (\widetilde{X})$ in the smooth and in the general case, respectively. Each stratum is shown to be a quasi-projective variety. Our result shows that the arithmetic depth of $X$ lies between $h+2$ and $h+5$ where $h$ denotes the dimension of the vertex of $\widetilde{X}$. Moreover we precisely describe the location of the point $p \in {\mathbb P}^{r+1} _K$ according to the value of the arithmetic depth of $X$.

Finally, in Section 6 we apply our results to classify non-normal varieties of almost minimal degree via their arithmetic depth. Theorem~\ref{6.2 Theorem} gives a complete list of all pairs $(\widetilde{X},p)$ for which $X = \pi_p (\widetilde{X})$ is arithmetically Cohen-Macaulay and hence a non-normal Del Pezzo variety. If $X$ is not a cone, then $\widetilde{X}$ is either $S(a) \subset {\mathbb P}^{a}_K$ with $a \geq 3$, or $S(1,b) \subset {\mathbb P}^{b+3} _K$ with $b \geq 2$, or $S(2,b) \subset {\mathbb P}^{b+4} _K$ with $b \geq 2$, or $S(1,1,c) \subset {\mathbb P}^{c+4}_K$ with $c \geq 1$. This provides a complete picture of non-normal Del Pezzo varieties from the view point of linear projections and normalizations. Also this reproves a result of T. Fujita which states that if $X \subset \mathbb P^r_K$ is a non-normal Del Pezzo variety and not a cone, then the dimension of $X$ is $\leq 3$, while there is no upper bound on the degree of $X$. See $(2.9)$ in \cite{Fu2} and $(9.10)$ in \cite{Fu}. \\

\noindent {\bf Acknowledgement} : This paper was started when the second named author was conducting Post Doctorial Research at the Institute of Mathematics in the University of Zurich. He thanks to them for their hospitality. The first named author thanks to the KIAS in Seoul and the KAIST in Daejeon for their hospitality and financial support offered during the preparation of this paper.
\vspace{0.3 cm}

\section{Preliminaries}
\label{2. Preliminaries}

\begin{notation and remark} \label{2.1 Notation and Remarks}
Let $\tilde X \subseteq {\mathbb P}^{r + 1}_K$ be a variety of minimal degree. That is, $\tilde X \subseteq {\mathbb P}^{r + 1}_K$ is a non-degenerate
irreducible projective subvariety of ${\mathbb P}^{r + 1}_K$ such that $\deg (\tilde X) = \codim (\tilde X) + 1$. According to the well-known classification of varieties of minimal degree (cf. \cite[Theorem 19.9]{Ha}), $\tilde X \subseteq {\mathbb P}^{r + 1}_K$ is either
\vspace{0.1 cm}

\renewcommand{\descriptionlabel}[1]%
             {\hspace{\labelsep}\textrm{#1}}
\begin{description}
\setlength{\labelwidth}{13mm}
\setlength{\labelsep}{1.5mm}
\setlength{\itemindent}{0mm}

\item[ ] (i) $\mathbb P^{r+1}_K$;
\vspace{0.1 cm}

\item[ ] (ii) a quadric hypersurface;
\vspace{0.1 cm}

\item[ ] (iii) (a cone over) the Veronese surface in ${\mathbb P} ^5_K$; or
\vspace{0.1 cm}

\item[ ] (iv) a rational normal scroll.
\end{description}
\vspace{0.1 cm}

\noindent In particular $\tilde X$ is always arithmetically Cohen-Macaulay. \qed
\end{notation and remark}
\vspace{0.1 cm}

\begin{notation and remark} \label{2.2 Notation and Remarks}
Let $\tilde X \subseteq {\mathbb P}^{r + 1}_K$ be as above and let $p \in {\mathbb P}^{r + 1}_K \setminus \tilde
X$ be a fixed closed point. The union of all secant lines to $X$ passing through $p$ is called the \textit{secant cone of} $\tilde X$
with respect to $p$ and denoted by $\Sec _p(\tilde X)$, i.e.
\begin{equation*}
\Sec _p(\tilde X):= \bigcup_{q \in \widetilde{X},~ \mbox{length}({\tilde X} \cap \langle p , q \rangle ) \geq 2} ~ \langle p , q \rangle .
\end{equation*}
We use the convention that $\Sec _p(\tilde X)= \{ p \}$ when there is no secant line to $X$ passing through $p$. Observe that $\Sec _p (\tilde X)$ is a cone with vertex $p$. We define the \textit{secant locus} $\Sigma _p(\tilde X)$ of $\tilde X$ with respect to $p$ as the scheme-theoretic intersection of $\widetilde{X}$ and $\Sec _p(\tilde X)$. Therefore
\begin{equation*}
\Sigma _p(\tilde X)_{\mbox{red}} = \{ q \in \tilde X ~|~ \mbox{length} ( \widetilde{X} \cap \langle p,q \rangle ) \geq 2 \} .
\end{equation*}
Thus $\Sigma _p(\tilde X)$ is the \textit{entry locus} of $\tilde X$ with respect to $p$ in the sense of \cite{Ru}.\qed
\end{notation and remark}
\vspace{0.1 cm}

\begin{notation and remarks} \label{2.3 Notation and Remarks}
(A) Let $\tilde X \subseteq {\mathbb P} ^{r + 1}_K$ and $p \in {\mathbb P}^r_K \setminus \tilde X$ be as
above. We fix a projective space ${\mathbb P}^r_K$ and consider the linear projection
\begin{equation*}
\pi _p : \widetilde{X} \rightarrow X_p:= \pi _p(\tilde X) \subseteq {\mathbb P}^r_K
\end{equation*}
of $\widetilde{X}$ from $p$. As the morphism $\pi_p : \widetilde{X} \rightarrow X_p$ is finite, we have
\begin{equation*}
\codim (\tilde X) = \codim (X_p) + 1 \leq \deg (X_p)
      \big\arrowvert \deg (\tilde X) = \codim (\tilde X) + 1 .
\end{equation*}
If $\codim (\tilde X) = 1$, then $\tilde X$ is a quadric and $\pi_p$ is a double covering of $\mathbb P^r _K$. On the other hand, if $\codim (\tilde X) > 1$ then
$\pi_p$ is birational and $X_p \subseteq {\mathbb P}^r_K$ is of almost minimal degree $($that is $\deg (X_p) = \codim (X_p) + 2)$.
\vspace{0.1 cm}

\noindent (B) Assume now, that $\codim (\tilde X) > 1$. Let $\depth (X_p)$ denote the {\it arithmetic depth of $X_p$}. Then, according to
\cite[Theorem 1.3]{B-Sche}, there are the following two cases:
\vspace{0.1 cm}

\renewcommand{\descriptionlabel}[1]%
             {\hspace{\labelsep}\textrm{#1}}
\begin{description}
\setlength{\labelwidth}{13mm}
\setlength{\labelsep}{1.5mm}
\setlength{\itemindent}{0mm}

\item[(2.1)] {\it If $X_p$ is smooth, then $\tilde X$ is smooth and $\pi_p : \tilde X \rightarrow X_p$ is an isomorphism. In this case, $\depth (X_p) = 1$ and $X_p$ is non-linearly normal.}
\item[(2.2)] {\it If $X_p$ is singular, then $\Sec _p(\tilde X) = {\mathbb P}^{t - 1}_K \subset \mathbb P^{r+1} _K$ for some $t \geq 2$ and $\Sigma _p(\tilde X) \subseteq {\mathbb P} ^{t - 1}_K$ is a hyperquadric. In this case, $\pi_p(\Sigma _p(\tilde X)) = {\mathbb P} ^{t - 2}_K$ is the non-normal locus of $X_p$ and
  \begin{equation*}
  \depth (X_p) = t.
  \end{equation*}
Therefore $X_p$ is non-normal and linearly normal.}
\end{description}

\noindent (C) Suppose that $\widetilde{X}$ is smooth. Then $\mbox{\rm Sec}(\tilde X)$ will denote the secant variety of $\tilde X$, i.e. the closure of the union of chords joining pairs of distinct points of $\tilde X$. Also $\mbox{\rm Tan} (\tilde X)$ will denote the tangent variety of $\tilde X$, i.e. the closure of the union of the tangent spaces of $\tilde X$. Thus $X_p$ is smooth if and only if $p \notin \mbox{\rm Sec}(\tilde X)$. \qed
\end{notation and remarks}
\vspace{0.3 cm}

\section{Possible Secant Loci} \label{5. Possible Secant Loci}
Throughout this section we keep the previously introduced notation. Let $\tilde X \subseteq {\mathbb P}^{r + 1}_K$ be a rational normal scroll. The aim of this section is to show which secant loci $\Sigma _p (\tilde X) \subseteq \mbox{Sec}_p (\tilde X) = {\mathbb P}^{t - 1} _K$ may occur at all.

We first treat the case in which $\tilde X$ is a smooth rational normal scroll.

\begin{notation and remark} \label{3.1 Notation and Remarks}
(A) Let $\widetilde{X} \subset {\mathbb P}^{r+1} _K$ be a smooth rational normal scroll. Thus there is a projection morphism $\varphi : \widetilde{X} \rightarrow {\mathbb P}^1 _K$. For each $x \in \mathbb P^1 _K$, let $\mathbb{L} (x)=\varphi^{-1}(x)$ denote the ruling of $\widetilde{X}$ over $x$.\\
(B) Let $q_1 , q_2$ be two distinct points in $\mathbb L (x)$ for some $x \in \mathbb P^1 _K$. According to Remark 7.4.(B) in \cite{B-Sche} we have $T_{q_1} \tilde X \cap T_{q_2} \tilde X = \mathbb L (x)$. \qed
\end{notation and remark}
\vspace{0.1 cm}

\begin{theorem}\label{thm:classificationsecantloci}
Let $\widetilde{X} \subset {\mathbb P}^{r+1} _K$ be a smooth rational normal scroll and of codimension at least $2$ and let
$p \in \mbox{\rm Sec}(\widetilde{X}) \setminus \widetilde{X}$. Then either
\begin{enumerate}
\item[$(a)$] $\mbox{Sec}_p (\widetilde{X}) = {\mathbb P}^1 _K$ and
$\Sigma_p (\widetilde{X}) \subset {\mathbb P}^1 _K$ is either a double
point or the union of two simple points.
\item[$(b)$] $\mbox{Sec}_p (\widetilde{X}) = {\mathbb P}^2 _K$ and $\Sigma_p (\widetilde{X}) \subset {\mathbb P}^2 _K$ is either a smooth conic or the union of a line $L$ which is contained in a ruling $\mathbb{L} (x) \subset \widetilde{X}$ and a line section $L'$ of $\widetilde{X}$.
\item[$(c)$] $\mbox{Sec}_p (\widetilde{X}) = {\mathbb P}^3 _K$ and $\Sigma_p (\widetilde{X}) \subset {\mathbb P}^3 _K$ is a smooth quadric surface.
\end{enumerate}
\end{theorem}
\vspace{0.1 cm}

\begin{proof}
According to $(2.1)$ and $(2.2)$, $\mbox{Sec}_p (\widetilde{X}) = {\mathbb P}^{t-1} _K$ is a linear subspace for some $t \geq 2$ and $\Sigma_p (\widetilde{X}) \subset {\mathbb P}^{t-1} _K$ is a hyperquadric. So, if $t=2$, we get statement $(a)$. Therefore we may assume that $t \geq 3$.

Let $f : \Sigma_p (\widetilde{X}) \rightarrow {\mathbb P}^1 _K$ be the restriction of $\varphi$ to $\Sigma_p (\widetilde{X})$. Clearly $f$ is surjective since otherwise $\Sigma_p (\widetilde{X}) \subset \mathbb{L}(x)$ for some $x \in {\mathbb P}^1 _K$ and so $p \in \langle  \Sigma_p (\widetilde{X}) \rangle \subset \mathbb{L}(x)$, which contradicts the assumption that $p \notin \widetilde{X}$.

Our next claim is that $t \leq 4$. Assume to the contrary that $t \geq 5$. As $f$ is surjective, for any two distinct points $x_1,x_2 \in {\mathbb P}^1 _K$, the two fibers $f^{-1} (x_i)$ are $(t-3)$-dimensional projective algebraic sets in $\mbox{Sec}_p (\widetilde{X}) = {\mathbb P}^{t-1} _K$, and so
\begin{equation*}
\mbox{dim}~ f^{-1} (x_1) \cap f^{-1} (x_2) \geq (t-3)+(t-3)-(t-1) = t-5 \geq 0.
\end{equation*}
This is impossible since $f^{-1} (x_1)$ and $f^{-1} (x_2)$ are disjoint. This contradiction shows that $t=3$ or $4$.

We now claim that $\Sigma_p (\widetilde{X})$ cannot be a double linear space. If $\Sigma_p (\widetilde{X}) = 2E$ where $E \subseteq {\mathbb P}^{t-1}_K$
is a $(t-2)$ dimensional linear space, then for any $q \in E$ we have $p \in T_q \widetilde{X}$. Clearly $\varphi |_E : E \rightarrow {\mathbb P}^1$ is a constant map if $t =4$. Thus there are the following two possibilities:
\begin{enumerate}
\item[$(i)$] $t \geq 3$ and $\varphi |_E : E \rightarrow {\mathbb P}^1 _K$
is a constant map;
\item[$(ii)$] $t=3$ and $\varphi |_E : E \rightarrow {\mathbb P}^1 _K$ is surjective
(and hence an isomorphism).
\end{enumerate}
In case (i), choose two distinct points $q_1$ and $q_2$ in $E$.
Then, by Notation and Remarks \ref{3.1 Notation and Remarks}.(B) we have
\begin{equation*}
p \in T_{q_1} \widetilde{X} \cap T_{q_2} \widetilde{X} = \mathbb{L}(x) \subset \widetilde{X},
\end{equation*}
which contradicts the assumption that $p \notin \widetilde{X}$. Now consider
case (ii). Then $E$ is a line section of $\varphi : \widetilde{X} \rightarrow {\mathbb P}^1 _K$. Choose two distinct points $q_1$ and $q_2$ in $E$ and let $x_i = \varphi (q_i)$ for $i=1,2$. Since $\langle p,E\rangle $ is contained in $T_{q_2} \widetilde{X}$, the linear space $\langle p,E\rangle  \cap \mathbb{L}(x_2)$ is of positive dimension. Now the equalities
\begin{equation*}
T_{q_1} \widetilde{X} = \langle  \mathbb{L}(x_1),q_2\rangle = \langle \mathbb{L}(x_1),\langle p,E\rangle  \cap \mathbb{L}(x_2)\rangle
\end{equation*}
imply that $\mbox{dim}~T_{q_1} \widetilde{X} \geq n+1$, which is a contradiction. So, $\Sigma_p (\widetilde{X})$ cannot be a double linear space.

If $t=3$ and hence $\mbox{Sec}_p (\widetilde{X}) = {\mathbb P}^2 _K$, then $\Sigma_p (\widetilde{X}) \subset {\mathbb P}^2 _K$ is either the union of two distinct lines $L_1$ and $L_2$ or else a smooth plane conic curve. Consider the first
case. Let $L_1 \cap L_2 = \{ q \}$ and $x = \varphi (q) \in {\mathbb P}^1 _K$. For each $i=1,2$, either $L_i \subset \mathbb{L}(x)$ or else $L_i$ is a line section of $ \widetilde{X}$. If $L_1$ and $L_2$ are both contained in $\mathbb{L}(x)$, then
\begin{equation*}
p  \in \mbox{Sec}_p (\widetilde{X}) = \langle  L_1,L_2\rangle  \subset \mathbb{L}(x) \subset \widetilde{X}
\end{equation*}
which contradicts our assumption that $p \notin \widetilde{X}$. If $L_1$ and $L_2$
are both line sections of $\widetilde{X}$, then $L_1 \cap L_2 = \emptyset$, a contradiction. Therefore, one of the two lines
$L_1$ and $L_2$ is contained in $\mathbb{L}(x)$ and the other one is a line section of $\widetilde{X}$. So, if $t=3$ we are precisely in the situation of statement $(b)$.

Finally let $t =4$. Then we have $\mbox{Sec}_p (\widetilde{X}) = {\mathbb P}^3 _K$. Since $\Sigma_p (\widetilde{X})$ is not a double plane, it is either the union of two distinct planes $E_1$ and $E_2$ or else irreducible. In the first case, $\varphi |_{E_i} : E_i \rightarrow {\mathbb P}^1$ is a constant map for $i=1,2$. Since $E_1 \cap E_2 \neq \emptyset$, this implies that $\Sigma_p (\widetilde{X})$ is contained in $\mathbb{L}(x)$ for some $x \in {\mathbb P}^1 _K$. Therefore
\begin{equation*}
p \in \mbox{Sec}_p (\widetilde{X}) = \langle \Sigma_p (\widetilde{X})\rangle  \subset \mathbb{L}(x) \subset \widetilde{X},
\end{equation*}
which contradicts to the assumption that $p \notin \widetilde{X}$. So $\Sigma_p (\widetilde{X})$ is irreducible. Assume that $\Sigma_p (\widetilde{X})$ is singular. Then $\Sigma_p (\widetilde{X})$ is a cone over a smooth plane conic $F$. If $q$
denotes the vertex point of $\Sigma_p (\widetilde{X})$, the family of lines $\{ L_{\lambda} := \langle q,\lambda\rangle  ~|~\lambda \in F\}$ covers $\Sigma_p (\widetilde{X})$. Since $q \in L_{\lambda}$ for all $\lambda \in F$, there exists at most one $\lambda \in F$ such that $L_{\lambda}$ is a line section of $\widetilde{X}$.
Thus for general $\lambda \in F$, the line $L_{\lambda}$ is contained in $\mathbb{L}(x)$ where $x = \varphi (q)$. But this implies that $\Sigma_p (\widetilde{X}) \subset \mathbb{L}(x)$. Again it follows that $p \in \widetilde{X}$, a contradiction. This completes the proof that $\Sigma_p (\widetilde{X})$ is a smooth quadric surface if $t = 4$.
\end{proof}
\vspace{0.1 cm}

We now consider the case in which the scroll $\tilde X$ is not necessarily smooth. First we introduce some notation.
\vspace{0.1 cm}

\begin{notation and remark}\label{3.2 Notation and Remarks}
Let $\tilde X \subseteq {\mathbb P}^{r + 1}_K$ be a rational normal scroll of codimension at least $2$ and with the vertex $\mbox{Vert}(\tilde X) = {\mathbb P}^h_K$ for some $h \geq -1$. Let $\mbox{dim}~ \widetilde{X} = n+h+1$ and note that $\widetilde{X}$ is a cone over an $n$-fold rational normal scroll $\tilde X_0$ in $\langle \tilde X_0 \rangle = {\mathbb P}^{r - h}_K$. Consider the projection map
\begin{equation*}
\psi : {\mathbb P}^{r + 1}_K \setminus \mbox{Vert}(\tilde X) \twoheadrightarrow \langle \tilde X_0 \rangle = \mathbb P ^{r-h}_K.
\end{equation*}
For a closed point $p \in {\mathbb P}^{r + 1}_K \backslash \mbox{Vert}(\tilde X)$, we denote $\psi (p)$ by $\overline{p}$. \qed
\end{notation and remark}
\vspace{0.1 cm}

\begin{theorem} \label{3.3 theorem}
Assume that the rational normal scroll $\tilde X \subseteq {\mathbb P}^{r + 1}_K$ has codimension at least $2$ and
let $p \in {\mathbb P}^{r + 1}_K \backslash \tilde X$. Then we have either
\vspace{0.1 cm}

\renewcommand{\descriptionlabel}[1]%
             {\hspace{\labelsep}\textrm{#1}}
\begin{description}
\setlength{\labelwidth}{13mm}
\setlength{\labelsep}{1.5mm}
\setlength{\itemindent}{0mm}

\item[{\rm (a)}] $\mbox{\rm Sec}_p(\tilde X) = \langle \mbox{\rm Vert}(\tilde X), p \rangle = {\mathbb P}^{h + 1}$ and $\Sigma _p
(\tilde X) = 2\mbox{\rm Vert}(\tilde X) \subseteq {\mathbb P}^{h + 1} _K$.
\vspace{0.1 cm}

\item[{\rm (b)}] $\mbox{\rm Sec}_p(\tilde X) = \langle \mbox{\rm Vert}(\tilde X), L \rangle = {\mathbb P}^{h + 2}_K$
for some line $L \subseteq \langle \tilde X_0 \rangle$ and either
\vspace{0.1 cm}

\renewcommand{\descriptionlabel}[1]%
             {\hspace{\labelsep}\textrm{#1}}
\begin{description}
\setlength{\labelwidth}{10mm}
\setlength{\labelsep}{1.2mm}
\setlength{\itemindent}{0mm}
\item[\rm (i)] $\Sigma _p(\tilde X) =
\mbox{\rm Join}(\mbox{\rm Vert}(\tilde X), Z) \subseteq {\mathbb P}
^{h + 2}_K$, where $Z \subseteq L$ consists of two simple points; or

\item[\rm (ii)] $\Sigma _p(\tilde X) = 2 \langle \mbox{\rm Vert}
(\tilde X), Z \rangle \subseteq {\mathbb P}^{h + 2}_K$, where
$Z \subseteq L$ consists of one simple point.

\end{description}
\vspace{0.1 cm}

\item[{\rm (c)}] $\mbox{\rm Sec}_p(\tilde X) = \langle \mbox{\rm
Vert}(\tilde X), P \rangle = {\mathbb P}^{h + 3}_K$
for some plane $P \subseteq \langle \tilde X_0 \rangle $
and $\Sigma _p(\tilde X) =
\mbox{\rm Join}(\mbox{\rm Vert}(\tilde X), W)$, where $W \subseteq
P$ is either a smooth conic or the union of two lines $L,
L' \subseteq P$ such that $L \subseteq {\mathbb L}(x)
\cap \tilde X_0$ for some $x \in {\mathbb P}^1_K$ and $L'$ is a line
section of $\tilde X_0$.
\vspace{0.1 cm}

\item[{\rm (d)}] $\mbox{\rm Sec}_p(\tilde X) = \langle \mbox{Vert}
(\tilde X), D \rangle = {\mathbb P}^{h + 4}_K$ for
some $3$-space $D \subseteq \langle \tilde X_0 \rangle$ and
$\Sigma _p(\tilde X) = \mbox{\rm
Join}(\mbox{\rm Vert}(\tilde X), V)$, where $V \subseteq D$ is a
smooth quadric surface.
\end{description}
\end{theorem}

\begin{proof}
In view of Theorem~\ref{thm:classificationsecantloci} we may assume that $\mbox{Vert}(\tilde X)
\not= \emptyset $, that is $h \geq 0$. We first show that
\vspace{0.1 cm}

\renewcommand{\descriptionlabel}[1]%
             {\hspace{\labelsep}\textrm{#1}}
\begin{description}
\setlength{\labelwidth}{13mm}
\setlength{\labelsep}{1.5mm}
\setlength{\itemindent}{0mm}

\item[(3.1)] {\it $\mbox{\rm Sec}_p(\tilde X) = \langle \mbox{\rm
Vert}(\tilde X), \mbox{\rm Sec}_{\overline p} (\tilde X_0) \rangle$}
\end{description}
and
\vspace{0.1 cm}

\renewcommand{\descriptionlabel}[1]%
             {\hspace{\labelsep}\textrm{#1}}
\begin{description}
\setlength{\labelwidth}{13mm}
\setlength{\labelsep}{1.5mm}
\setlength{\itemindent}{0mm}

\item[(3.2)] {\it $\Sigma _p(\tilde X)_{\mbox{\rm red}} =
\mbox{\rm Join}(\mbox{\rm Vert}(\tilde X), \Sigma _{\overline p}(\tilde X_0) )$.}
\end{description}
\vspace{0.1 cm}

\noindent Indeed a line $L = {\mathbb P}^1_K \subseteq {\mathbb P}^{r + 1}_K \backslash
\mbox{Vert}(\tilde X)$ is secant to $\tilde X$ if and only if $\overline {L} := \psi (L)$ is secant to $\tilde X _0$ since $\tilde X \cap L$ and $\tilde X_0 \cap \overline{L}$ are isomorphic via $\psi$. This shows (3.1). For the second statement, remember that
   \[ \Sigma _p(\tilde X)
      _{\mbox{\rm red}} = (\mbox{\rm Sec}_p(\tilde X) \cap \tilde X)_{\rm red}
   \]
and that
embedded joins are reduced. Since $\mbox{\rm Sec}_p(\tilde X) = \langle
\mbox{\rm Vert}(\tilde X), \mbox{\rm Sec}_{\overline p} (\tilde X_0) \rangle$
and $\tilde X = \mbox{Join}(\mbox{Vert}(\tilde X), \tilde X _0)$, we obtain
\begin{equation*}
\begin{CD}
\Sigma _p(\tilde X)
_{\rm red} & \quad = \quad  & \big[ \langle \mbox{Vert} (\tilde X), \mbox{Sec}
_{\overline p} (\tilde X_0) \rangle \cap \mbox{Join}(\mbox{Vert}(\tilde X),
\tilde X_0) \big] _{\rm red} \\
& \quad = \quad  & \mbox{Join}(\mbox{Vert}(\tilde X), \tilde X_0 \cap \mbox{Sec}
_{\overline p}(\tilde X_0)) = \mbox{Join}(\mbox{Vert}(\tilde X),
\Sigma _{\overline p}(\tilde X)).
\end{CD}
\end{equation*}
\vspace{0.1 cm}

\noindent Now our assertion is shown by combining (3.1), (3.2) and Theorem \ref{thm:classificationsecantloci}.
\end{proof}
\vspace{0.3 cm}

\section{The Secant Stratification in the Smooth Case}\label{4. The Secant Stratification in the Smooth Case}

According to Theorem \ref{thm:classificationsecantloci} there are six different
possibilities for the secant locus $\Sigma _p(\tilde X)$ of the smooth rational
normal scroll $\tilde X \subseteq {\mathbb P}^{r + 1}_K$ with respect
to the point $p \in {\mathbb P}^{r + 1}_K \backslash \tilde X$. This gives a decomposition of ${\mathbb P}^{r + 1}_K \backslash
\tilde X$ into six disjoint strata. The aim of this section is to describe this stratification in geometric terms.
\vspace{0.1 cm}

\begin{definition and remark}\label{4.1 Definition and Remark}
(A) (cf. Example 8.26 in \cite{Ha}) Let
\begin{equation*}
\tilde X = S(a_1 , \ldots , a_n) \subseteq {\mathbb P}^{r + 1}_K
\end{equation*}
be an $n$-dimensional smooth rational normal scroll of type $(a_1 , \ldots , a_n)$ with $1 \leq a_1 \leq \ldots \leq a_n$. We assume that $\mbox{codim}(\widetilde{X}) \geq 2$, or equivalently, $a_1 + \ldots + a_n \geq 3$. We write
$$k := \begin{cases} \max \{ i \in \{ 1, \ldots , n \} \big\arrowvert a_i = 1 \} & \mbox{if $a_1 = 1$, and} \\
0     & \mbox{if $a_1 >1$}\end{cases}$$
and
$$m := \begin{cases} \max \{ i \in \{ k + 1, \ldots , n \} \big\arrowvert a_i = 2 \} & \mbox{if $a_{k+1} =2$, and} \\
k & \mbox{if $a_{k + 1} \neq 2$.} \end{cases}$$
So, we have $k \in \{ 0, \ldots , n \} $ and $m \in \{ k, \ldots , n \} $ and may write
\begin{equation*}
\tilde X = S(\underbrace{1, \ldots , 1}_{k} ,  \underbrace{2, \ldots , 2}_{m - k} , a_{m + 1}, \ldots , a_n) = S(\underline 1,
\underline 2, \underline a)
\end{equation*}

\noindent with $3 \leq a_{m + 1} \leq \ldots \leq a_n$.
\vspace{0.1 cm}

\noindent (B) Consider the (possibly empty) scrolls
\vspace{0.1 cm}

\begin{equation*}
S(\underline 1) = S(\underbrace{1, \ldots , 1}_{k}) \subset \mathbb P^{2k-1} _K \quad \mbox{and} \quad S(\underline 2) = S(\underbrace{2, \ldots , 2}_{m-k}) \subset \mathbb P^{3m-3k-1} _K
\end{equation*}

\noindent which are contained in $\widetilde{X}$. Obviously $S(\underline 1) \cong \mathbb P^1 _K \times \mathbb P^{k-1} _K$ and $S(\underline 2) \cong C \times \mathbb P^{m-k-1} _K$ for a smooth plane conic $C \subset \mathbb P^2 _K$. For each $\alpha \in \mathbb P^{k-1} _K$, we denote the line $\mathbb P^1 _K \times \{ \alpha \}$ in $S(\underline 1)$ by $L_{\alpha}$. Similarly, for each $\beta \in \mathbb P^{m-k-1} _K$, we denote the smooth plane conic $C \times \{ \beta \}$ in $S(\underline 2)$ by $C_{\beta}$. Also $S(\underline 2)$ can be regarded as a subvariety of the Segre variety
\begin{equation*}
\bigtriangleup := \dot \bigcup _{\beta \in
	  {\mathbb P}^{m - k - 1}_K} \langle C_\beta \rangle = {\mathbb P}^2_K \times {\mathbb
	  P}^{m - k - 1}_K \subset \mathbb P^{3m-3k-1} _K .
\end{equation*}

\noindent (C) We define the following sets in ${\mathbb P}^{r + 1}_K \backslash \tilde X$:

\renewcommand{\descriptionlabel}[1]%
             {\hspace{\labelsep}\textrm{#1}}
\begin{description}
\setlength{\labelwidth}{13mm}
\setlength{\labelsep}{1.5mm}
\setlength{\itemindent}{0mm}

\item[(4.1)] $SL^\emptyset (\tilde X):= \{ p \in {\mathbb P}^{r + 1}_K
\backslash \tilde X \big\arrowvert \Sigma _p(\tilde X) =
\emptyset \}$;
\vspace{0.1 cm}

\item[(4.2)] $SL^{q_1, q_2}(\tilde X):= \{ p \in {\mathbb P}^{r + 1}_K
\backslash \tilde X \big\arrowvert \Sigma _p(\tilde X)$
consists of two simple points $\}$;
\vspace{0.1 cm}

\item[(4.3)] $SL^{2q}(\tilde X):= \{ p \in {\mathbb P}^{r + 1}_K
\backslash \tilde X \big\arrowvert \Sigma _p(\tilde X)$
is a double point in some straight line ${\mathbb P}^1_K \subseteq {\mathbb P}
^{r + 1}_K$ $\}$;
\vspace{0.1 cm}

\item[(4.4)] $SL^{L_1 \cup L_2}(\tilde X):= \{ p \in {\mathbb P}^{r + 1}_K
\backslash \tilde X \big\arrowvert \Sigma _p(\tilde X)$
is the union of two distinct coplanar simple lines $\}$;
\vspace{0.1 cm}

\item[(4.5)] $SL^C(\tilde X):= \{ p \in {\mathbb P}^{r + 1}_K
\backslash \tilde X \big\arrowvert \Sigma _p(\tilde X)$ is a smooth plane
conic $\}$;
\vspace{0.1 cm}

\item[(4.6)] $SL^Q(\tilde X):= \{ p \in {\mathbb P}^{r + 1}_K
\backslash \tilde X \big\arrowvert \Sigma _p(\tilde X)$ is a smooth quadric
in a $3$-space $\}$.
\end{description}
\vspace{0.1 cm}

\noindent According to Theorem \ref{thm:classificationsecantloci} we have
\vspace{0.1 cm}

\renewcommand{\descriptionlabel}[1]%
             {\hspace{\labelsep}\textrm{#1}}
\begin{description}
\setlength{\labelwidth}{13mm}
\setlength{\labelsep}{1.5mm}
\setlength{\itemindent}{0mm}

\smallskip

\item[(4.7)  ] ${\mathbb P}^{r + 1}_K =$
\begin{equation*}
\quad \quad \quad \tilde X {\dot \cup } SL^\emptyset (\tilde X) {\dot \cup} SL^{q_1, q_2}(\tilde X) {\dot \cup } SL^{2q}(\tilde
X) {\dot \cup } SL^{L_1 \cup L_2}(\tilde X) {\dot \cup } SL^C(\tilde X) {\dot \cup } \ SL^Q(\tilde X).
\end{equation*}
\end{description}
This decomposition will be called \textit{the secant stratification of} $\widetilde{X} \subset \mathbb P^{r+1} _K$. \qed
\end{definition and remark}
\vspace{0.1 cm}

We now describe the strata (4.1) - (4.6).
\vspace{0.1 cm}

\begin{theorem}\label{4.2 Theorem} Assume that the rational normal scroll $\tilde X
\subseteq {\mathbb P}^{r + 1}_K$ is of codimension at least $2$ and
smooth. Let (cf. Definition and Remark \ref{4.1 Definition and Remark}.(B))
   \begin{align*} &A:= \langle S(\underline 1) \rangle; \\
      &B := \mbox{\rm Join}(S(\underline{1}),\widetilde{X});\\
	  &U:= \mbox{\rm Join}(A, \bigtriangleup).
   \end{align*}
Then $\tilde X \subseteq B, \ A \subseteq B \cap U, \ B\cup U \subseteq
\mbox{\rm Tan}(\tilde X)$ and

\renewcommand{\descriptionlabel}[1]%
             {\hspace{\labelsep}\textrm{#1}}
\begin{description}
\setlength{\labelwidth}{13mm}
\setlength{\labelsep}{1.5mm}
\setlength{\itemindent}{0mm}

\smallskip

\item[{\rm (a)}] $SL^\emptyset (\tilde X) = {\mathbb P}^{r + 1}_K
\backslash \mbox{\rm Sec}(\tilde X)$;

\smallskip

\item[{\rm (b)}] $SL^{q_1, q_2}(\tilde X) = \mbox{\rm Sec}(\tilde X)
\backslash \mbox{\rm Tan}(\tilde X)$;

\smallskip

\item[{\rm (c)}] $SL^{2q}(\tilde X) = \mbox{\rm Tan}(\tilde X)
\backslash (B \cup U)$;

\smallskip

\item[{\rm (d)}] $SL^{L_1 \cup L_2}(\tilde X) = B \backslash (A \cup
\tilde X)$;

\smallskip

\item[{\rm (e)}] $SL^C(\tilde X) = U \backslash B$;

\smallskip

\item[{\rm (f)}] $SL^Q (\tilde X) = A \backslash \tilde X$.
\end{description}
\end{theorem}
\vspace{0.1 cm}

In order to prove this theorem we will need a preliminary result we mention now:

\begin{lemma}\label{lem:4.1}
Let $\widetilde{X} \subset {\mathbb P}^{r+1} _K$ be a smooth rational normal scroll and let $p$ be a closed point contained in $\mathbb P^{r+1} _K \setminus \widetilde{X}$. \\
$(1)$ If $\mbox{dim} ~\Sigma_p (\widetilde{X}) >0$, then $p \in \mbox{Tan} (\widetilde{X})$.\\
$(2)$ Let $\widetilde{X} = S(\underbrace{1, \ldots , 1}_{n})\subset {\mathbb P}^{2n-1} _K$ with $n \geq 2$. Then $\Sigma_p (\widetilde{X})$ is a smooth quadric surface.\\
$(3)$ Let $\widetilde{X} = S(1,2)\subset {\mathbb P}^4 _K$. Then $\mbox{dim}~\Sigma_p (\widetilde{X})=1$. \\
$(4)$ Let $\widetilde{X} = S(\underbrace{1,\cdots,1}_{n-1},2)\subset {\mathbb P}^{2n} _K$ with $n \geq 2$. Then $\mbox{dim}~\Sigma_p (\widetilde{X}) \geq 1$.
\end{lemma}

\begin{proof}
$(1)$: If $\codim (\tilde X) \leq 1$, our claim is obvious. So, let $\codim (\tilde X) \geq 2$. If $\mbox{dim} \Sigma_p (\widetilde{X}) >0$, then $\Sigma_p (\widetilde{X})$
is either an irreducible hyperquadric in the linear space $\mbox{Sec}_p
(\widetilde{X})$ or else the union of a line $L$ which is contained in a
ruling $\mathbb{L} (x) \subset \widetilde{X}$ and a line section $L'$ of $\widetilde{X}$ by Theorem \ref{thm:classificationsecantloci}. In the first case, clearly $p \in \mbox{Tan}(\Sigma_p (\widetilde{X}))
\subset \mbox{Tan}(\widetilde{X})$. In the latter case, let $L \cap L' = \{ q \}$.
Therefore $T_q \widetilde{X} = \langle \mathbb{L}(x), L'\rangle $ contains
the plane $\mbox{Sec}_p (\widetilde{X})=\langle L,L'\rangle $. In particular,
$p \in T_q \widetilde{X} \subset \mbox{Tan}(\widetilde{X})$.

\noindent $(2), (3),(4)$: If $n=2$, statement $(2)$ is clear. So we may assume that $\codim (\tilde X) \geq 2$. Remember that $\dim \Sigma _p(\tilde X) \leq 2$ by Theorem \ref{thm:classificationsecantloci}. The statements will be proved by applying the following inequalities (cf. \cite[I.(1.3.3)]{Ru} and \cite[I.6.3. Theorem 7]{Sha}):
\renewcommand{\descriptionlabel}[1]%
             {\hspace{\labelsep}\textrm{#1}}
\begin{description}
\setlength{\labelwidth}{13mm}
\setlength{\labelsep}{1.5mm}
\setlength{\itemindent}{0mm}

\smallskip

\item[(4.8)  ] $(2n+1) - \mbox{dim}~ \mbox{Sec}(\widetilde{X}) \leq \mbox{dim}~ \Sigma_p (\widetilde{X})$
\end{description}
When $\widetilde{X} = S(1,\cdots,1)$, (4.8) implies that $\mbox{dim}~ \Sigma_p (\widetilde{X}) \geq 2$ since $\mbox{Sec}(\widetilde{X})$ is a subset of ${\mathbb P}^{2n-1} _K$. Therefore $\Sigma_p (\widetilde{X})$ is a smooth quadric surface by Theorem \ref{thm:classificationsecantloci}. When $\widetilde{X} = S(1,2)$, obviously $\dim \Sigma _p(\tilde X) \leq 1$ while (4.8) implies that $\dim \Sigma _p(\tilde X) \geq 1$. Thus we get $\mbox{dim} ~\Sigma_p (\widetilde{X}) = 1$. Similarly, for $\widetilde{X} = S( 1,\cdots,1 ,2)$, (4.8) enables us to conclude that $\mbox{dim}~\Sigma_p (\widetilde{X}) \geq 1$.
\end{proof}
\vspace{0.1 cm}

Now we give the \\

\noindent {\bf Proof of Theorem \ref{4.2 Theorem}.} The inclusions $\tilde X \subseteq B$ and $A \subseteq
U$ are obvious. As $A = \mbox{Sec}(S(\underline 1))$ we get $A  \subseteq \mbox{Join}(S(\underline 1), \tilde X) = B$, hence $A \subseteq B \cap U$.

To prove the inclusions $B \subseteq \mbox{Tan}(\tilde X)$ and $U \subseteq \mbox{Tan}(\tilde X)$ we write
\begin{equation*}
S(\underline 1) = \bigcup _{\alpha \in {\mathbb P} _K^{k - 1}} L_{\alpha} \quad \mbox{and} \quad  \bigtriangleup = \bigcup_{\beta \in {\mathbb P}^{m - k - 1}_K} \langle C_{\beta}\rangle
\end{equation*}
(cf. Definition and Remark \ref{4.1 Definition and Remark}.(B)). Now, the first inclusion
follows from the equalities

\begin{enumerate}
\item[(4.9)] $\quad B \quad = \quad \mbox{Join}(S(\underline{1}),\widetilde{X})$ \par
$\hskip 0.9 cm = \quad \bigcup_{\alpha \in {\mathbb P}^{k - 1}_K} \mbox{Join}(L_{\alpha},\widetilde{X})$ \par
$\hskip 0.9 cm = \quad \bigcup_{L_{\alpha \in {\mathbb P}^{k-1}_K}} \mbox{Join}(L_{\alpha},\bigcup_{x \in {\mathbb P}^1_K} \mathbb{L}(x))$ \par
$\hskip 0.9 cm = \quad \bigcup_{\alpha \in {\mathbb P}^{k - 1}_K,~ x \in {\mathbb P}^1_K} \langle L_{\alpha},\mathbb{L}(x)\rangle $.
\end{enumerate}

\noindent To prove the second inclusion we first observe that

\begin{enumerate}
\item[(4.10)] $\quad U \quad = \quad \mbox{Join}(A,\bigtriangleup)$ \par
$\hskip 0.9 cm = \quad \mbox{Join}(A,\bigcup_{\beta \in {\mathbb P}^{m - k - 1}_K} \langle C_{\beta}\rangle ) $ \par
$\hskip 0.9 cm = \quad \bigcup _{\beta \in {\mathbb P}^{m - k - 1}_K} \langle A,\langle C_{\beta}\rangle \rangle .$
\end{enumerate}

\noindent For each $\beta \in {\mathbb P}^{m - k - 1}_K$ consider the smooth
rational normal scroll
\begin{equation*}
S(\underline 1,2)_{\beta} := \langle A,\langle C_{\beta}\rangle \rangle  \cap \widetilde{X} \subset \langle A,\langle C_{\beta}\rangle \rangle .
\end{equation*}
Since $S(\underline 1,2)_{\beta}$ spans the linear space $\langle A,\langle C_{\beta}\rangle \rangle $, Lemma \ref{lem:4.1}.$(4)$ implies that
\begin{equation*}
\langle A,\langle C_{\beta}\rangle \rangle  = \mbox{Tan} ~ S(\underline 1,2)_{\beta}  \subset \mbox{Tan} (\widetilde{X}).
\end{equation*}
This gives the desired inclusion $U \subseteq \mbox{Tan}(\tilde X)$.
\vspace{0.1 cm}

We now prove statements (a), (b), (f), (d), (e) and (c).
\vspace{0.1 cm}

(a): This follows by Notation and Remarks \ref{2.3 Notation and Remarks}.(C).
\vspace{0.1 cm}

(b): ``$\subseteq $'': Let $p \in SL^{q_1,q_2} (\widetilde{X})$.
Then $\Sigma_p (\widetilde{X}) = \{q_1,q_2\}$ and $L:=\langle q_1,q_2\rangle $
is a secant line to $\widetilde{X}$ and hence $p \in \mbox{Sec}(\widetilde{X})$.
Now assume that $p \in \mbox{Tan}(\widetilde{X})$. Then there exists $q \in \tilde X$
such that $p \in T_q \widetilde{X}$. Therefore $q \in \Sigma_p (\widetilde{X})$.
This implies that $q = q_1$ or $q=q_2$. In particular, $L$ is a tri-secant line to $\widetilde{X}$ and so we get the contradiction
that $p \in L \subseteq \widetilde{X}$. Therefore $p \in \mbox{Sec}(\widetilde{X}) \setminus \mbox{Tan}(\widetilde{X})$.
\vspace{0.1 cm}

$``\supseteq $'': Let $p \in \mbox{Sec}(\widetilde{X}) \setminus \mbox{Tan}
(\widetilde{X})$. By statement $(a)$
we have $\Sigma_p (\widetilde{X}) \neq \emptyset$. By Lemma \ref{lem:4.1}.$(1)$ we have $\mbox{dim} \Sigma_p (\widetilde{X}) \leq 0$.
Thus $\Sigma_p (\widetilde{X})$ has dimension zero. Now
Theorem \ref{thm:classificationsecantloci}.$(a)$ and  $p \notin \mbox{Tan}
(\widetilde{X})$ guarantee that $\Sigma_p
(\widetilde{X})$ is the union of two simple points.
\vspace{0.1 cm}

(f): ``$\subseteq $'': Let $p \in SL^Q (\widetilde{X})$ so that $\mbox{Sec}_p (\widetilde{X}) = {\mathbb P}^3_K$ and $\Sigma_p (\widetilde{X})
\subset {\mathbb P}^3_K$ is a smooth quadric surface. Remember that $\Sigma_p (\widetilde{X})$ contains two disjoint families of lines $\{L_{\lambda}\}$, $\{ M_{\lambda}\}$, each parameterized by $\lambda \in
{\mathbb P}^1_K$. Also since $p \notin \widetilde{X}$, the restriction map
$\varphi |_{\Sigma_p (\widetilde{X})} : \Sigma_p (\widetilde{X}) \rightarrow
{\mathbb P}^1_K$ is surjective. Therefore one of the two families, say $\{L_{\lambda}\}$,
consists of line sections of $\widetilde{X}$. Thus $\Sigma_p (\widetilde{X})
\subset S(\underline{1})$ and hence $p \in \mbox{Sec}_p (\widetilde{X}) = \ \langle \Sigma_p (\widetilde{X})\rangle  \ \subset
A \ = \ \langle S(\underline{1})\rangle $.
\vspace{0.1 cm}

``$\supseteq $'': Let $p \in A \setminus \widetilde{X}$. Then $k >1$ (s. Definition and Remark \ref{4.1 Definition and Remark}.(B)). Therefore Lemma \ref{lem:4.1}.$(2)$ shows that $\Sigma_p (S(\underline{1}))$ is a smooth quadric surface.
Since $\Sigma_p (S(\underline{1})) \subset \Sigma_p (\widetilde{X})$
this implies that $\mbox{dim}~\Sigma_p (\widetilde{X}) \geq 2$. Now by Theorem \ref{thm:classificationsecantloci} we see that $\Sigma_p (\widetilde{X})=\Sigma_p (S(\underline{1}))$ is a smooth quadric surface and hence $p \in SL^Q (\widetilde{X})$.
\vspace{0.1 cm}

(d): ``$\subseteq $'': Let $p \in SL^{L_1 \cup L_2} (\widetilde{X})$ so that $\mbox{Sec}_p (\widetilde{X}) = {\mathbb P}^2_K$ and $\Sigma_p (\widetilde{X}) =L_1 \cup L_2$ where $L_i$ are lines such that $L_1$ is a line section of $\widetilde{X}$. Thus $L_1 \subset S(\underline{1})$ and $L_2 \subset {\mathbb L}(x)$ for some $x \in {\mathbb P}^1_K$. This shows that $p \in \langle L_1,L_2\rangle =\mbox{Sec}_p (\widetilde{X}) \subset \mbox{Join} (S(\underline{1}),\widetilde{X})=B$. On the other hand statement (f) implies $p \notin A$. Therefore, $p \in B \setminus (A \cup \widetilde{X})$.
\vspace{0.1 cm}

``$\supseteq $'': Let $p \in B \setminus (A \cup \widetilde{X})$. By (4.8),
there exists a line $L_{\alpha } \subset S(\underline{1})$ and a point
$x \in {\mathbb P}^1_K$ such that $p \in \langle L_{\alpha},\mathbb{L}(x)\rangle  = {\mathbb P}^n_K$. Let $L$ denote the line $\langle p,L_{\alpha}\rangle  \cap \mathbb{L}(x)$. Then clearly $L_{\alpha} \cup L \subset \Sigma_p (\widetilde{X})$. On the other hand, Theorem \ref{thm:classificationsecantloci} and statement (f) show that $\Sigma_p (\widetilde{X})$ has dimension at most one so that $\Sigma_p (\widetilde{X}) = L_{\alpha} \cup L$ and hence $p \in SL^{L_1 \cup L_2} (\widetilde{X})$.
\vspace{0.1 cm}

(e): ``$\subseteq $'': Let $p \in SL^C (\widetilde{X})$. This means that $\mbox{Sec}_p (\widetilde{X}) = {\mathbb P}^2_K$ and $\Sigma_p (\widetilde{X}) \subset {\mathbb P}^2_K$ is a smooth plane conic curve. Clearly $\varphi |_{\Sigma_p (\widetilde{X})} : \Sigma_p (\widetilde{X}) \rightarrow  {\mathbb P}^1_K$ is a surjective map. We will show that indeed $\Sigma_p (\widetilde{X})$ is a conic section of $\tilde X$, or equivalently, that $\varphi |_{\Sigma_p (\widetilde{X})}$ is bijective. Suppose to the contrary
that there is a point
$x \in  {\mathbb P}^1_K$ such that $\mathbb{L}(x)$ meets $\Sigma_p (\widetilde{X})$ in two distinct points $q_1$ and $q_2$. Then the two lines $T_{q_1} \Sigma_p (\widetilde{X})$ and $T_{q_2} \Sigma_p (\widetilde{X})$ meet at a point $z \in \mbox{Sec}_p (\widetilde{X})
\backslash \tilde X$. Thus we have
\begin{equation*}
T_{q_1} (\widetilde{X}) = \langle  \mathbb{L}(x),z\rangle  = T_{q_2} (\widetilde{X}),
\end{equation*}
which is impossible by Notation and Remarks \ref{3.1 Notation and Remarks}.(B). Therefore,
$\Sigma_p (\widetilde{X})$ is a conic section of $\tilde X$. In particular
$\Sigma_p (\widetilde{X}) \subset S(\underline{1},\underline{2})$.

Next we claim that $\mbox{Sec}_p (\widetilde{X}) \cap A = \emptyset$. Suppose
to the contrary that there is a closed point $z$ in $\mbox{Sec}_p (\widetilde{X}) \cap A$. If $z \notin S(\underline{1})$, then $k>1$ and so $\mbox{Sec}_z (\widetilde{X})$ contains the $3$-dimensional space $\mbox{Sec}_z (S (\underline{1}))$ (s. Lemma \ref{lem:4.1}.$(2)$) and the point $p$.
This implies that $\mbox{Sec}_z (\widetilde{X})$ has dimension at least $4$, a
contradiction to Theorem~\ref{thm:classificationsecantloci}.
If $z \in S(\underline{1})$, then let $L_z$ be the unique line section of $\widetilde{X}$ which passes through $z$. Note that $z \in L_z \cap \Sigma_p (\widetilde{X})$ since otherwise $\langle p,z \rangle$ is a proper tri-secant line to $\widetilde{X}$, which is not possible since $\widetilde{X}$ is cut out by quadrics. As $\Sigma_p (\tilde X)$ is a smooth plane conic curve, we have $L_z \nsubseteq \mbox{Sec}_p (\tilde X) = \langle \Sigma_p (\tilde X) \rangle$, since otherwise $\tilde X$ would have tri-secant lines.
Observe that the $3$-dimensional linear space $\langle \mbox{Sec}_p (\widetilde{X}) , L_z \rangle$ contains both $p$ and the surface
\begin{equation*}
S := \overline{\bigcup_{x \in \mathbb P^1 _K \setminus \{ \varphi (z) \}} \langle \mathbb L (x) \cap \Sigma_p (\widetilde{X}) , \mathbb L (x) \cap L_z \rangle} ~ \subset ~ \tilde X.
\end{equation*}
As the conic $\Sigma_p (\tilde X)$ is contained in $S$, $S$ cannot be a plane. So the generic line in $\langle\mbox{Sec}_p (\tilde X), L_z \rangle$ passing through $p$ is a secant line to $S$, whence to $\tilde X$. It follows that $\Sigma_p (\widetilde{X})$ contains $S$, which contradicts the fact that $\dim ~ \Sigma_p (\tilde X) =1$. This completes the proof that $\mbox{Sec}_p (\widetilde{X}) \cap A = \emptyset$.

Now consider the canonical projection map
   \[ \pi : \langle S(\underline{1},\underline{2})\rangle
      \setminus A \rightarrow \langle S(\underline{2})\rangle
   \]
which fixes $\langle S(\underline 2)
\rangle $. The image $\pi (\Sigma _p(\tilde X))$ is contained in $S(\underline{2})$. Also it is again a smooth plane conic as $\mbox{Sec}_p(\tilde X) \cap A = \emptyset $.  Moreover, for all $x \in {\mathbb P}^1_K$
we have $\pi ({\mathbb L}(x) \cap (\langle S(\underline 1, \underline 2) \rangle
\backslash A)) \subseteq {\mathbb L}(x)$ so that $\sharp (\pi (\Sigma _p
(\tilde X)) \cap {\mathbb L}(x)) \leq 1$ for all such $x$. Therefore $\pi
(\Sigma _p(\tilde X))$ is a conic section of $\tilde X$
and hence is equal to $C_{\beta}$ for some $\beta \in \mathbb P^{m-k-1} _K$. Therefore $\pi (\mbox{Sec}_p
(\widetilde{X})) = \langle C_{\beta}\rangle $ which guarantees that
\begin{equation*}
p \in \mbox{Sec}_p (\widetilde{X}) \subset  \mbox{Join} (A,\langle C_{\beta}\rangle ) \subset U.
\end{equation*}
On the other hand, $p \notin (B \setminus (A \cup \widetilde{X})) \cup (A \setminus \widetilde{X}) \cup \widetilde{X} = B$ by statements (d)
and (f). This completes the proof that $p \in U \setminus B$.
\vspace{0.1 cm}

``$\supseteq $'': Let $p \in U \setminus B$. By statements (a), (b), (d), (f) and
Theorem~\ref{thm:classificationsecantloci} we have
\begin{equation*}
p \in SL^{2q} (\widetilde{X}) \dot{\cup} SL^C (\widetilde{X}).
\end{equation*}
Thus it suffices to show that $\mbox{dim}~\Sigma_p (\widetilde{X}) \geq 1$.
By $(4.9)$, there exists a point
$\beta \in {\mathbb P}^{m - k - 1}_K$
such that $p \in \langle A,\langle C_{\beta}\rangle \rangle  \setminus A$.
If $p \in \langle C_{\beta}\rangle $, then $C_{\beta} \subset \Sigma_p
(\widetilde{X})$ and hence $\mbox{dim}~\Sigma_p (\widetilde{X}) \geq 1$.
Now assume that $p \notin \langle C_{\beta}\rangle$, so that $A \neq \emptyset$, and consider
the canonical projection map (which fixes $A$)
\begin{equation*}
\varrho : \langle A,\langle C_{\beta}\rangle \rangle  \setminus \langle C_{\beta}\rangle  \rightarrow A.
\end{equation*}
Let $q = \varrho (p)$. If $q \in S(\underline{1})$, let $L$ be the unique line section of $\widetilde{X}$ which passes through $q$. Then
\begin{equation*}
p \in \mbox{Join} (L,\langle C_{\beta}\rangle )= \mathbb P^4_K.
\end{equation*}
Moreover, $\mathbb P^4_K$ contains the smooth rational normal surface scroll
\begin{equation*}
S(1,2) = \bigcup_{v \in L, ~ w \in C_{\beta},~ \varphi (v)=
\varphi (w)} \langle  v,w\rangle  \quad \subset \quad  \widetilde{X}.
\end{equation*}
Since $\Sigma_p (S(1,2)) \subset \Sigma_p (\widetilde{X})$, Lemma \ref{lem:4.1}.$(3)$ shows that $\mbox{dim}~\Sigma_p (\widetilde{X}) \geq 1$. If $q \in A \setminus S(\underline{1})$, we have $k>1$ and Lemma \ref{lem:4.1}.$(2)$ implies that $\Sigma_q (S(\underline{1}))$ is a smooth quadric
surface. Moreover
\begin{equation*}
p \in \mbox{Join} (\langle \Sigma_q (S(\underline{1}))\rangle , \langle C_{\beta} \rangle )=
\mathbb P^6_K.
\end{equation*}
For each point $x \in \mathbb P^1_K$
consider the line $L_x =  \Sigma_q (S(\underline{1})) \cap \mathbb{L}(x)$. Now
$\mathbb P^6_K$ contains the threefold rational normal scroll
\begin{equation*}
S(1,1,2) = \bigcup_{x \in \mathbb P^1 _K, ~w \in C_{\beta}, \varphi (w)=x}
\langle L_x , w\rangle  \subset \widetilde{X}.
\end{equation*}
As $\Sigma_p (S(1,1,2)) \subset \Sigma_p (\widetilde{X})$, it remains to show that $\mbox{dim}~\Sigma_p (S(1,1,2)) \geq 1$. This follows by Lemma \ref{lem:4.1}.$(4)$.
\vspace{0.1 cm}

(c): It is easy to see that ${\mathbb P}^{r + 1}_K \backslash \tilde X$
is the disjoint union of the sets ${\mathbb P}^{r + 1}_K
\backslash \mbox{Sec}(\tilde X)$, $\mbox{Sec}(\tilde X) \backslash
\mbox{Tan}(\tilde X), \mbox{Tan}(\tilde X) \backslash (B \cup U),
B \backslash (A \cup \tilde X), U \backslash B$ and $A \backslash
\tilde X$. Now, statements (a), (b), (d), (e), (f) and the equality
(4.7) imply that $SL^{2q}(\tilde X) = \mbox{Tan}(\tilde X) \backslash
(B \cup U)$. \qed
\vspace{0.3 cm}

\section{The Secant Stratification in the General Case} \label{5. The Secant Stratification in the General Case}

We now treat the secant stratification in the general case, that is in
the case where the scroll $\tilde X$ is not necessarily smooth. We
keep all the previous hypotheses and notations.
\vspace{0.1 cm}

We first appropriately generalize the concepts defined in
Definition and Remark~\ref{4.1 Definition and Remark}.

\begin{definition and remark}
\label{5.1 Definition and Remark} (A) Let $\tilde X \subseteq {\mathbb P}^{r + 1}_K$ be a rational normal scroll of codimension at least $2$ and with vertex $\mbox{Vert}(\tilde X) =
{\mathbb P}^h_K$ for some $h \geq 0$. As in Notation and Remark \ref{3.2 Notation and Remarks}, let $\mbox{dim}~ \tilde X = n+h+1$ and note that $\widetilde{X}$ is a cone over an $n$-fold rational normal scroll $\tilde X_0$ in
$\langle \tilde X_0 \rangle = {\mathbb P}^{r - h}_K$. Let $(a_1 , \ldots , a_{n})$ be the type of $\tilde X_0$. Let the integers $k$ and $m$ and the subvarieties $S(\underline{1})$, $S(\underline{2})$ and $\bigtriangleup$ of ${\mathbb P}^{r - h}_K$ be as in Definition and Remark~\ref{4.1 Definition and Remark}. Again we consider the projection map (cf. Notation and Remark \ref{3.2 Notation and Remarks})
\begin{equation*}
\psi : {\mathbb P}^{r + 1}_K \setminus \mbox{Vert}(\tilde X) \twoheadrightarrow \langle \tilde X_0 \rangle = \mathbb P ^{r-h}_K
\end{equation*}
and write $\overline{p} := \psi (p)$ for a closed point $p \in {\mathbb P}^{r + 1}_K \backslash \mbox{Vert}(\tilde X)$.
\vspace{0.1 cm}

\noindent (B) We define the
following sets in ${\mathbb P}^{r + 1}_K \backslash \tilde X$:
\vspace{0.1 cm}

\renewcommand{\descriptionlabel}[1]%
             {\hspace{\labelsep}\textrm{#1}}
\begin{description}
\setlength{\labelwidth}{13mm}
\setlength{\labelsep}{1.5mm}
\setlength{\itemindent}{0mm}

\item[(5.1)] $SL^\emptyset (\tilde X):= \{ p \in {\mathbb P}^{r + 1}
_K \backslash \tilde X \big\arrowvert \Sigma _p(\tilde X) = 2\mbox{\rm
Vert} (\tilde X) \subseteq \langle \mbox{\rm Vert}(\tilde X), p \rangle
\} $;
\vspace{0.1 cm}

\item[(5.2)] $SL^{q_1, q_2} (\tilde X):= \{ p \in {\mathbb P}^{r + 1}
_K \backslash \tilde X \big\arrowvert \Sigma _p(\tilde X) = \mbox{\rm
Join}(\mbox{\rm Vert} (\tilde X), \{ x,y\} )$ where $x, y$ are two
distinct points in some line ${\mathbb P}^1_K \subseteq \langle \tilde
X_0 \rangle \} $;
\vspace{0.1 cm}

\item[(5.3)] $SL^{q} (\tilde X):= \{ p \in {\mathbb P}^{r + 1}
_K \backslash \tilde X \big\arrowvert \Sigma _p(\tilde X) = 2\mbox{\rm
Join}(\mbox{\rm Vert} (\tilde X),x) \subseteq $
$\mbox{\rm Join}
(\mbox{\rm Vert}(\tilde X), L)$, where $L \subseteq \langle \tilde X_0
\rangle $ is a line and $x \in L\} $;
\vspace{0.1 cm}

\item[(5.4)] $SL^{L_1 \cup L_2} (\tilde X):= \{ p \in {\mathbb P}^{r + 1}
_K \backslash \tilde X \big\arrowvert \Sigma _p(\tilde X) = \mbox{\rm
Join}(\mbox{\rm Vert} (\tilde X),L \cup L')$ where $L, L' \subseteq
\langle \tilde X_0 \rangle $ are two distinct coplanar simple lines $\} $;
\vspace{0.1 cm}

\item[(5.5)] $SL^{C} (\tilde X):= \{ p \in {\mathbb P}^{r + 1}
_K \backslash \tilde X \big\arrowvert \Sigma _p(\tilde X) = \mbox{\rm
Join}(\mbox{\rm Vert} (\tilde X),V)$ where $V \subseteq
\langle \tilde X_0 \rangle $ is a smooth plane conic $\} $;
\vspace{0.1 cm}

\item[(5.6)] $SL^{Q} (\tilde X):= \{ p \in {\mathbb P}^{r + 1}
_K \backslash \tilde X \big\arrowvert \Sigma _p(\tilde X) = \mbox{\rm
Join}(\mbox{\rm Vert} (\tilde X),W)$ where $W \subseteq
\langle \tilde X_0 \rangle $ is a smooth quadric surface in a $3$-space
$\} $.
\end{description}
Clearly, if $\tilde X$ is smooth, the strata defined in (5.1) - (5.6)
respectively coincide with the corresponding strata defined in (4.1) -
(4.6). According to Theorem~\ref{3.3 theorem} we have
\vspace{0.1 cm}

\renewcommand{\descriptionlabel}[1]%
             {\hspace{\labelsep}\textrm{#1}}
\begin{description}
\setlength{\labelwidth}{13mm}
\setlength{\labelsep}{1.5mm}
\setlength{\itemindent}{0mm}

\item[(5.7) ] ${\mathbb P}^{r + 1}_K = $
\begin{equation*}
\quad \quad \quad \tilde X \dot \cup SL^\emptyset (\tilde X) \dot \cup SL^{q_1, q_2}(\tilde X) \dot \cup SL^{2q}
(\tilde X) \dot \cup SL^{L_1 \cup L_2}(\tilde X) \dot \cup SL^C(\tilde X) \dot \cup SL^Q(\tilde X).
\end{equation*}
\end{description}
This decomposition will be called \textit{the secant stratification of} $X \subset \mathbb P^{r+1} _K$. \qed
\end{definition and remark}
\vspace{0.1 cm}

\begin{lemma}
\label{5.2 Lemma} Let $p \in {\mathbb P}^{r + 1}_K \backslash \tilde X$
and let $\overline p \in \langle \tilde X_0 \rangle $ be defined
according to Definition and Remark \ref{5.1 Definition and Remark}.(A). Then
\begin{equation*}
p \in SL^\ast (\tilde X) \Longleftrightarrow \overline
p \in SL^\ast (\langle \tilde X_0 \rangle )
\end{equation*}
where $\ast $ runs through the set of symbols $\{ \emptyset , (q_1, q_2), 2q, L_1 \cup L_2, C, Q\}$.
\end{lemma}
\vspace{0.1 cm}

\begin{proof}
Clear from Theorem~\ref{3.3 theorem}.
\end{proof}
\vspace{0.1 cm}

\begin{theorem}\label{5.3 Theorem}
Let $\tilde X \subseteq {\mathbb P}^{r + 1}_K$ be a rational normal scroll of codimension $\geq 2$. In the notations of Definition and Remark \ref{5.1 Definition and Remark}, let
 \begin{align*} &V:= \mbox{\rm Join}(\mbox{\rm Vert} (\widetilde{X}), \mbox{\rm Tan}(\tilde X_0));\\
      &W:= \mbox{\rm Join}(\mbox{\rm Vert} (\widetilde{X}), \mbox{\rm Sec} (\tilde X_0));\\
      &A:= \langle \mbox{\rm Vert} (\widetilde{X}), \langle S(\underline 1) \rangle \rangle; \\
      &B := \mbox{\rm Join}(\mbox{\rm Vert} (\widetilde{X}), \mbox{\rm Join}(S(\underline 1), \tilde X_0));\\
	  &U:= \mbox{\rm Join}(A, \bigtriangleup).
   \end{align*}
Then $\tilde X \subseteq B, A \subseteq B, B \cup U \subseteq V \subseteq W$ and
\vspace{0.1 cm}

\renewcommand{\descriptionlabel}[1]%
             {\hspace{\labelsep}\textrm{#1}}
\begin{description}
\setlength{\labelwidth}{13mm}
\setlength{\labelsep}{1.5mm}
\setlength{\itemindent}{0mm}

\item[\rm (a)] $SL^\emptyset (\tilde X) = {\mathbb P}^{r + 1}_K \backslash W$;
\vspace{0.1 cm}

\item[\rm (b)] $SL^{q_1, q_2} (\tilde X) = W \backslash V$;
\vspace{0.1 cm}

\item[\rm (c)] $SL^{2q} (\tilde X) = V \backslash (B \cup U)$;
\vspace{0.1 cm}

\item[\rm (d)] $SL^{L_1 \cup L_2} (\tilde X) = B \backslash (A \cup \tilde X)$;
\vspace{0.1 cm}

\item[\rm (e)] $SL^{C} (\tilde X) = U \backslash B$;
\vspace{0.1 cm}

\item[\rm (f)] $SL^{Q} (\tilde X) = A \backslash \tilde X$.
\end{description}
\end{theorem}
\vspace{0.1 cm}

\begin{proof}
For any closed subset $T \subseteq \langle \tilde X_0\rangle $, we have $\psi ^{-1}(T) = \mbox{\rm Join}(\mbox{\rm Vert} (\widetilde{X}), T) \backslash
Z$. Therefore we get the relations
\begin{align*} &\psi ^{-1}(\tilde X_0) & = \quad & \tilde X \backslash \mbox{\rm Vert} (\widetilde{X}) ;\\
			   &\psi ^{-1}(\langle S(\underline 1) \rangle ) &=\quad& A \backslash \mbox{\rm Vert} (\widetilde{X}); \\
			   &\psi ^{-1}(\mbox{Join} (S(\underline 1), \tilde X_0)) &=\quad& B \backslash \mbox{\rm Vert} (\widetilde{X});\\
			   &\psi ^{-1} (\mbox{Join}(\langle S(\underline 1), \bigtriangleup )) &=\quad& U \backslash \mbox{\rm Vert} (\widetilde{X}); \\
			   &\psi ^{-1}(\mbox{Tan}(\tilde X_0)) &=\quad& V \backslash \mbox{\rm Vert} (\widetilde{X}) \quad \mbox{and} \\
               &\psi ^{-1}(\mbox{Sec}(\tilde X_0)) &=\quad& W \backslash \mbox{\rm Vert} (\widetilde{X}).
\end{align*}
Thus we get our claim by combining Theorem \ref{4.2 Theorem} and Lemma~\ref{5.2 Lemma}.
\end{proof}
\vspace{0.1 cm}

\begin{remark}\label{5.4 remark}
Let $Z = \mbox{\rm Vert} (\widetilde{X})$ and let $p \in {\mathbb P}^{r + 1}_K \backslash \tilde X$. By Notation and Remarks \ref{2.3 Notation and Remarks} and Theorem \ref{5.3 Theorem}, the secant loci $\Sigma _p(\tilde X), \Sigma _{\overline p}(\tilde X_0)$ and the arithmetic depth of the projection $X_p \subset \mathbb P^r _K$
depend on the position of $p$ as shown by the following table
\bigskip
\[
\renewcommand{\arraystretch}{1.25}
\begin{tabular}{| c | c | c | c | c |}\hline
$p\in  $&$\Sigma _{\overline p}(\tilde X_0) \subseteq \tilde X_0$
   & $\Sigma _p(\tilde X) \subseteq \tilde X$  & $\dim (\Sigma _p(\tilde X))$
   & $\mbox{\rm depth}(X_p)$ \\ \hline
${\mathbb P}^{r + 1}_K \backslash W$  & $\emptyset $ & $2Z \subseteq \langle Z,
   p \rangle $  &  $h$  & $ h + 2$ \\  \hline
$W \backslash V$ & $\{ q_1, q_2\} $ & $\langle Z, q_1 \rangle \cup \langle
   Z, q_2 \rangle $ & $h + 1$ & $h + 3$ \\
& $q_1, q_2 \in X_0, q_1 \not= q_2$ &&& \\ \hline
$V \backslash (B \cup U)$ & $2q \in \langle \overline p, q \rangle $ & $2
   \langle Z, q \rangle \subseteq \langle Z, p, q \rangle $ & $ h + 1$ & $h
   + 3$ \\
& $q \in \tilde X_0$ &&& \\ \hline
$B \backslash (A \cup \tilde X)$ & $L_1 \cup L_2 \subseteq \tilde X_0 $ &
   $\langle Z, L_1 \rangle \cup \langle Z, L_2 \rangle $ & $ h + 2$ & $h
   + 4$ \\
& $\mbox{\rm two coplanar lines}$ &&& \\ \hline
$U \backslash B$ & $C \subseteq \tilde X_0 \mbox{ \rm a }$ &$\mbox{\rm Join}
   (Z,C)$ & $h + 2$ & $h + 4$ \\
& $\mbox{\rm smooth plane conic}$ &&& \\ \hline
$A \backslash \tilde X$ &$Q \subseteq \tilde X_0, \mbox{ \rm a smooth }$
   & $\mbox{\rm Join}(Z,Q)$ & $h + 3$ & $h + 5 $ \\
& $\mbox{\rm quadric surface}$ &&& \\ \hline
\end{tabular}
\]
\centerline{\rm Table 1: The Secant Stratification}
\end{remark}
\vspace{0.3 cm}

\section{Non-Normal Del Pezzo Varieties}\label{6. Non-Normal Del Pezzo Varieties}

\begin{remark} \label{6.1 Remark} (A) A non-degenerate closed integral subscheme $X \subseteq {\mathbb P}^r_K$ is called a {\it maximal Del Pezzo variety} if it is arithmetically Cohen-Macaulay and satisfies $\deg (X) = \codim (X) + 2$ (cf. \cite[Definition 6.3]{B-Sche}).
A {\it Del Pezzo variety} is a projective variety $X \subseteq {\mathbb P}^r_K$ which is an isomorphic
projection of a maximal Del Pezzo variety. It is equivalent to say that the polarized pair $(X, {\mathcal O}_X(1))$ is Del Pezzo
in the sense of Fujita \cite{Fu}, (cf. \cite[Theorem 6.8]{B-Sche}). So in particular we can say:
\vspace{0.1 cm}

\renewcommand{\descriptionlabel}[1]%
             {\hspace{\labelsep}\textrm{#1}}
\begin{description}
\setlength{\labelwidth}{13mm}
\setlength{\labelsep}{1.5mm}
\setlength{\itemindent}{0mm}

\item[(6.1)] {\it A Del Pezzo variety $X \subseteq {\mathbb P}^r_K$ is
maximally Del Pezzo if and only if it is linearly normal}.
\end{description}
\vspace{0.1 cm}

\noindent (B) Let $X \subseteq {\mathbb P}^r_K$ be a non-degenerate closed integral subscheme. According to \cite[Theorem 6.9]{B-Sche} and Notation and Remark \ref{2.3 Notation and Remarks},
\vspace{0.1 cm}

\renewcommand{\descriptionlabel}[1]%
             {\hspace{\labelsep}\textrm{#1}}
\begin{description}
\setlength{\labelwidth}{13mm}
\setlength{\labelsep}{1.5mm}
\setlength{\itemindent}{0mm}

\item[(6.2)] {\it $X \subseteq {\mathbb P}^r_K$ is a non-normal maximal Del Pezzo variety if and only if $X = \pi_p (\widetilde{X})$ where $\tilde X \subseteq {\mathbb P}^{r + 1}_K$ is of minimal degree with $\codim (\widetilde{X}) \geq 2$, $p$ is a closed point in ${\mathbb P}^{r + 1}_K \backslash \tilde X$ with $\dim~ \Sigma_p (\widetilde{X}) = \dim~(\widetilde{X}) -1$. }
\end{description}
\vspace{0.1 cm}

\noindent Obviously $\widetilde{X}$ is either (a cone over) the Veronese surface in $\mathbb P^5 _K$ or a rational normal scroll (cf. Notation and Remarks \ref{2.1 Notation and Remarks}). We say that the Del Pezzo variety $X$ is {\it exceptional} (resp. {\it non-exceptional}) if $\tilde X$ is (a cone over) the Veronese surface (resp. a rational normal scroll). \qed
\end{remark}
\vspace{0.1 cm}

In view of (6.2) it suffices to classify the pairs $(\widetilde{X}, p)$ where $X_p = \pi_p (\tilde X) \subseteq {\mathbb P}^r_K$ is arithmetically Cohen-Macaulay in order to classify the
non-normal maximal Del Pezzo varieties in ${\mathbb P}^r_K$. This we do now for the case of non-exceptional non-normal maximal Del Pezzo varieties. For the exceptional case, see Remark \ref{6.3 Remark}.
\vspace{0.1 cm}

\begin{theorem}\label{6.2 Theorem}
Let $\tilde X \subseteq {\mathbb P}^{r + 1}_K$ be a rational normal scroll with $\codim (\widetilde{X}) \geq 2$ and with vertex $Z := \mbox{\rm Vert} (\widetilde{X})=\mathbb P^h _K$ for some $h \geq -1$. Let $\tilde X_0 \subset \mathbb P^{r-h} _K$ be a smooth rational normal scroll such that $\tilde X = \mbox{\rm Join} (Z, \tilde X_0 )$. Then for a closed point $p$ in ${\mathbb P}^{r + 1}_K \backslash \tilde X$, the variety
\begin{equation*}
X_p := \pi_p (\widetilde{X}) \subseteq {\mathbb P}^r_K
\end{equation*}
is arithmetically Cohen-Macaulay, and hence non-normal maximally Del Pezzo precisely in the following cases:
\vspace{0.1 cm}

\renewcommand{\descriptionlabel}[1]%
             {\hspace{\labelsep}\textrm{#1}}
\begin{description}
\setlength{\labelwidth}{13mm}
\setlength{\labelsep}{1.5mm}
\setlength{\itemindent}{0mm}

\item[{\rm (a)}] $\tilde X_0 = S(a)$ for some integer $a > 2$ and
$p \in \mbox{\rm Join}(Z, \mbox{\rm Sec}(\tilde X_0))
\backslash \tilde X$.

\item[{\rm (b)}] {\rm (i)} $\tilde X_0 = S(1,2)$ and $p \in {\mathbb P}
^{r + 1}_K \backslash \tilde X$;

\item[  ] {\rm (ii)} $\tilde X_0 = S(1,b)$ with $b > 2$ and
$p \in \mbox{\rm Join}(Z, \mbox{\rm Join}(S(1), \tilde X_0))
\backslash \tilde X$;

\item[  ] {\rm (iii)} $\tilde X_0 = S(2,2)$ and $p \in \mbox{\rm Join}
(Z,\bigtriangleup) \backslash
\tilde X$;

\item[  ] {\rm (iv)} $\tilde X_0 = S(2,b)$ for some integer $b > 2$ and
$p \in \mbox{\rm Join}(Z, \langle S(2) \rangle ) \backslash
\tilde X$.

\item[{\rm (c)}] {\rm (i)} $\tilde X_0 = S(1,1,1)$ and $p \in
{\mathbb P}^{r + 1}_K \backslash \tilde X$;

\item[  ] {\rm (ii)} $\tilde X_0 = S(1,1,c)$ for some integer $c > 1$
and $p \in \mbox{\rm Join}(Z, \langle S(1,1) \rangle )
\backslash \tilde X$.
\end{description}
\end{theorem}

\begin{proof}
Let $n+h+1$ denote the dimension of $\tilde X$. Thus $\tilde X_0$ has dimension $n$. According to (6.2), $X_p$ is arithmetically Cohen-Macaulay if and only if $\Sigma_p (\widetilde{X})$ has dimension equal to $n+h$. Since $\dim~\Sigma_p (\widetilde{X}) = h+j ~(0 \leq j \leq 3)$ (cf. Theorem \ref{3.3 theorem} and Remark \ref{5.4 remark}), it follows that $X_p$ is arithmetically Cohen-Macaulay if and only if $n=j$ with $j \in \{ 0,1, 2, 3 \} $. Obviously the case $j = 0$ cannot occur. Therefore $n \in \{ 1 , 2, 3 \}$.
\vspace{0.1 cm}

``(a)'': Assume that $n = 1$, so that $\tilde X_0 = S(a)$ for some integer $a \geq 3$. Then Remark \ref{5.4 remark} yields that
$\mbox{depth}(X_p) = h + 3$ if and only if
   \[ p \in (W \backslash V)
      \cup (V \backslash B \cup U) = W \backslash B \cup U.
   \]
Since $B = \tilde X$, $U = Z$ and $W = \mbox{Join}(Z, \mbox{Sec}(\tilde X_0))$, it follows that $X_p$ is arithmetically Cohen-Macaulay if and only if $p \in
\mbox{Join}(Z, \mbox{Sec}(\tilde X_0)) \backslash \tilde X$.
\vspace{0.1 cm}

``(b)'': Assume that $n = 2$, so that $\tilde X_0 = S(a,b)$ for some integers $a, b$ with $1 \leq a
\leq b$ and $b \geq 2$. Then Remark \ref{5.4 remark} yields that
$\mbox{depth}(X_p) = h + 4$ if and only if
   \[ p \in (B \backslash
      A \cup \tilde X) \cup (U \backslash B) = (B \cup U) \backslash (A \cup
      \tilde X) .
	  \leqno{(6.3)}
   \]
If $a \geq 3$, then we have $\bigtriangleup =S (\underline 1) = \emptyset $ and hence $B=\tilde X$ and $U = Z \subseteq \tilde X$, which leaves no possibility for $p$. Therefore $a \leq 2$.

Observe that for all $b \geq 2$ we have
   \[ A = \begin{cases} \mbox{Join}(Z, \langle S(1) \rangle ) ,
      &\mbox{if } a = 1 ; \\ Z , &\mbox{if } a = 2 . \end{cases}
	  \leqno{(6.4)}
   \]
As $\langle S(1) \rangle = S(1)$ it follows $A \subseteq \tilde X$ for all
$b \geq 2$. So, the condition (6.3) imposed on $p$ now simply becomes
   \[ p \in (B \cup U) \backslash \tilde X .
      \leqno{(6.5)}
   \]
Also we get for all $b \geq 2$
   \[ B = \begin{cases} \mbox{Join}(Z,\mbox{Join}(S(1), \tilde X_0)) ,
      &\mbox{if } a = 1 ; \\ \tilde X, &\mbox{if } a = 2 . \end{cases}
	  \leqno{(6.6)}
   \]
Observe that $\bigtriangleup$ is given by
   \[ \bigtriangleup = \begin{cases} \langle S(2) \rangle  ,
      &\mbox{if } a = 1 \mbox{ and } b = 2 ; \\
	  \emptyset , &\mbox{if } a = 1 \mbox{ and } b > 2 ; \\
	  {\mathbb P}^2_K \times {\mathbb P}^1_K , &\mbox{if }
	  a = b = 2 ; \\
	  \langle S(2) \rangle  , &\mbox{if } a = 2 < b .
	  \end{cases}
   \]
So, we get the following possibilities for $U = \mbox{Join}(Z,
\mbox{Join}(\langle S(\underline 1) \rangle , \bigtriangleup )$
    \[ U = \begin{cases} {\mathbb P}^{r + 1}_K ,
       &\mbox{if } a = 1 \mbox{ and } b = 2 ; \\
	   \mbox{Join}(Z, S(1)) \subseteq \tilde X ,
	   &\mbox{if } a = 1 \mbox{ and } b > 2 ; \\
	   \mbox{Join}(Z, {\mathbb P}^2_K \times {\mathbb P}^1_K) ,
	   &\mbox{if } a = b = 2 ; \\
	   \mbox{Join}(Z, \langle S(2) \rangle ) , &\mbox{if } a = 2 < b .
	   \end{cases}
	   \leqno{(6.7)}
    \]
Combining (6.6) and (6.7), we now get for $T:= (B \cup U) \backslash
\tilde X$ the following values
   \[ T = \begin{cases} {\mathbb P}^{r + 1}_K \backslash \tilde X ,
       &\mbox{if } a = 1 \mbox{ and } b = 2 ; \\
	   \mbox{Join}(Z, \mbox{Join}(S(1), \tilde X_0)) \backslash \tilde X
	   &\mbox{if } a = 1 \mbox{ and } b > 2 ; \\
	   \mbox{Join}(Z, {\mathbb P}^2_K \times {\mathbb P}^1_K) \backslash
	   \tilde X , &\mbox{if } a = b = 2 ; \\
	   \mbox{Join}(Z, \langle S(2) \rangle ) \backslash \tilde X ,
	   &\mbox{if } a = 2 < b .
	   \end{cases}
   \]
This proves statement (b).
\vspace{0.1 cm}

``(c)'': Assume that $n = 3$, so that $\tilde X_0 = S(a,b,c)$ for some integers $a, b, c$ with $1 \leq
a \leq b \leq c$. According to Remark \ref{5.4 remark} we have $\mbox{depth}(X_p) = h + 5$
if and only if $p \in A \backslash \tilde X$. If $a > 1$, we have $A = Z$, so that no choice for $p$ is left. Therefore
$a = 1$. Now, for $A$ we get the following possibilities:
    \[ A = \begin{cases} \langle Z, \langle \tilde X_0 \rangle \rangle =
	   {\mathbb P}^{r + 1}_K , &\mbox{if } a = b = c = 1 ; \\
	   \mbox{Join}(Z, \langle S(1,1) \rangle ) ,
	   &\mbox{if } a = b = 1 < c ; \\
	   \mbox{Join}(Z, \langle S(1) \rangle ) , &\mbox{if } a = 1 < b .
	   \end{cases}
    \]
If $a = 1 < b$, we have $A = \mbox{Join}(Z, S(1)) \subseteq \tilde X$,
so that no possibility is left for $p$. Therefore $b = 1$. This proves claim (c).
\end{proof}
\vspace{0.1 cm}

\begin{remark}\label{6.3 Remark}
In order to understand all non-normal Del Pezzo varieties it suffices now to know the exceptional cases in which $X = X_p \subseteq {\mathbb P}^r_K$ where $\tilde X \subseteq {\mathbb P}^{r + 1}_K$ is a cone over the Veronese surface $S \subset {\mathbb P}^5 _K$. \\
(A) Recall that $\mbox{\rm Sec} (S)$ is a cubic hypersurface. Let $p$ be a closed point in $\mathbb P^5 _K \setminus S$. It belongs to folklore that the following statements are equivalent:
\begin{enumerate}
\item[$(a)$] $p \in \mbox{\rm Sec} (S) \setminus S$.
\item[$(b)$] $\Sigma_p (S)$ is a smooth plane conic curve.
\item[$(c)$] $\pi_p (S) \subset \mathbb P^4 _K$ is a complete intersection of two quadrics and so it is arithmetically Cohen-Macaulay.
\end{enumerate}
Therefore the secant stratification of $S \subset {\mathbb P}^5 _K$ is
\vspace{0.1 cm}

\renewcommand{\descriptionlabel}[1]%
             {\hspace{\labelsep}\textrm{#1}}
\begin{description}
\setlength{\labelwidth}{13mm}
\setlength{\labelsep}{1.5mm}
\setlength{\itemindent}{0mm}

\item[(6.8)] ${\mathbb P}^5 _K = S \dot{\cup} SL^{C} (S) \dot{\cup} SL^{\emptyset} (S)$
\end{description}
\vspace{0.1 cm}

\noindent where $SL^C(S)$ is equal to $\mbox{\rm Sec}(S) \setminus S$.

\noindent (B) In the same way as in the proof of Theorem \ref{3.3 theorem} and Theorem \ref{5.3 Theorem}, one can get the secant stratification of
$\tilde X \subset \mathbb P^{r+1} _K$ from (6.8). In particular, $X_p \subset {\mathbb P}^r _K$ is a maximal Del Pezzo variety if and
only if $p \in \mbox{Join}(\mbox{Vert}(\tilde X),\mbox{\rm Sec}(S)) \backslash \tilde X$. \qed
\end{remark}
\vspace{0.1 cm}

\begin{remark}
(A) Let $\tilde X \subseteq {\mathbb P}^{r + 1}_K$ be a variety of minimal degree with $\codim (\tilde X) \geq 2$ and let $p$ be a closed point in $\mathbb P^{r+1} _K \setminus \tilde X$. As a consequence of Theorem \ref{6.2 Theorem} and Remark \ref{6.3 Remark}, we finally have a complete list of pairs $(\tilde X , p )$ for which $X_p$ is a non-normal Del Pezzo variety. This provides a complete picture of non-normal Del Pezzo varieties from the point of view of linear projections and normalizations.

\noindent (B) Let the notations and hypotheses be as in Theorem~\ref{6.2 Theorem} and Remark \ref{6.3 Remark}. Since $X_p = \pi_p (\tilde X)$ is a cone with vertex $\pi_p(\mbox{\rm Vert}( \tilde X))$, the non-normal Del Pezzo varieties which are not cones are the varieties described in Theorem~\ref{6.2 Theorem} and Remark \ref{6.3 Remark} for which $\mbox{\rm Vert}( \tilde X) \not= \emptyset $. More precisely, the non-normal maximal Del Pezzo varieties which are not cones are precisely the following ones:
\vspace{0.1 cm}

\renewcommand{\descriptionlabel}[1]%
             {\hspace{\labelsep}\textrm{#1}}
\begin{description}
\setlength{\labelwidth}{13mm}
\setlength{\labelsep}{1.5mm}
\setlength{\itemindent}{0mm}

\item[{\rm (i)}] Projections of a rational normal curve $S(a) \subseteq
{\mathbb P}^a_K$ with $a > 2$ from a point $p \in \mbox{\rm Sec}(S(a))
\backslash S(a)$.

\item[{\rm (ii)}] Projections of the Veronese surface $S \subseteq {\mathbb P}^5_K$ from a point $p \in \mbox{\rm Sec} (S) \backslash S$.

\item[{\rm (iii)}] Projections of a smooth cubic surface scroll $S(1,
2) \subseteq {\mathbb P}^4_K$ from a point $p \in {\mathbb P}^4_K
\backslash S(1,2)$.

\item[{\rm (iv)}] Projections of a smooth rational normal scroll
$S(1,b) \subseteq {\mathbb P}^{b + 2}_K$ with $b > 2$ from a point
$p \in \mbox{\rm Join}(S(1), S(1,b)) \backslash S(1,b)$.

\item[{\rm (v)}] Projections of a smooth quartic surface scroll
$S(2,2) \subseteq {\mathbb P}^5_K$ from a point $p \in {\mathbb P}^2_K
\times {\mathbb P}^1_K \backslash S(2,2)$.

\item[{\rm (vi)}] Projections of a smooth surface scroll $S(2,b) \subseteq
{\mathbb P}^{b + 3}_K$ with $b > 2$ from a point $p \in \langle S(2)
\rangle \backslash S(2,b)$.

\item[{\rm (vii)}] Projections of a smooth $3$-fold scroll $S(1,1,1)
\subseteq {\mathbb P}^5_K$ from a point $p \in {\mathbb P}^5_K
\backslash S(1,1,1)$.

\item[{\rm (viii)}] Projections of a smooth $3$-fold scroll $S(1,1,c)
\subseteq {\mathbb P}^{c + 4}_K$ with $c > 1$ from a point $p \in
\langle S(1,1) \rangle \backslash S(1,1,c)$.
\end{description}
Therefore if $X \subset \mathbb P^r_K$ is a non-normal Del Pezzo and is not a cone, then the dimension of $X$ is $\leq 3$ while there is no upper bound of the degree of $X$. This fact was first shown by T. Fujita (cf. $(2.9)$ in \cite{Fu2} and $(9.10)$ in \cite{Fu}). \qed
\end{remark}
\vspace{0.1 cm}

\begin{remark}\label{6.6 Remark}
Using the same method as above, one can indeed classify all varieties $X \subset {\mathbb P}^r _K$ of almost minimal degree and codimension $\geq 2$, via their arithmetic depth, as projections from rational normal scrolls of given numerical type and the position of the center of the projection. We shall give a detailed exposition of this in a later paper. \qed
\end{remark}
\vspace{0.3 cm}

\bibliographystyle{plain}

\vskip 1 cm

\author{
 \begin{tabular}{llll}
         Institut f\"ur Mathematik        &&Department of Mathematics \\
         Universit\"at Z\"urich           &&Korea University \\
         Winterthurerstrasse 190          &&Seoul 136-701 \\
         CH-8057 Z\"urich  && Republic of Korea \\
		 Schwitzerland & \\
        {\it email: } brodmann@math.uzh.ch
        &&{\it email: } euisungpark@korea.ac.kr
 \end{tabular}
}

\end{document}